\documentclass[11pt]{amsart}
\usepackage{psfrag,epsfig}
\usepackage{hyperref}
\numberwithin{figure}{section}
\newtheorem{theorem}{Theorem}[section]
\newtheorem{prop}[theorem]{Proposition}
\theoremstyle{definition}
\newtheorem{definition}[theorem]{Definition}
\newtheorem{example}[theorem]{Example}
\theoremstyle{remark}
\newtheorem{remark}[theorem]{Remark}
\numberwithin{equation}{section}

\def \H{{\mathbb H}}
\def \R{{\mathbb R}}
\def \Z{{\mathbb Z}}
\def \Q{{\mathbb Q}}
\def \C{{\mathbb C}}

\begin{document}

\title{On the frieze patterns by Conway and Coxeter}
\author{Anna Felikson,  Pavel Tumarkin}
\address{Department of Mathematical Sciences, Durham University, Mathematical Sciences \& Computer Science Building, Upper Mountjoy Campus, Stockton Road, Durham, DH1 3LE, UK}
\email{anna.felikson@durham.ac.uk, pavel.tumarkin@durham.ac.uk}

 \maketitle

The story of frieze patterns started in $1970$s with a paper by Coxeter~\cite{C}  and two short papers by Conway and Coxeter~\cite{CC}, one listing questions and the other providing brief answers. During the first three  decades it was followed by only a handful of papers observing some further properties of the same objects. Then Propp~\cite{Pr} and Caldero and Chapoton~\cite{CaCh} established connections between frieze patterns and the recent theory of cluster algebras introduced by Fomin and Zelevinsky in early $2000$s. Since then the theory of frieze patterns exploded with new results, and frieze patterns found a wide spectrum of connections in various domains of mathematics, including combinatorics, geometry, representation theory and integrable systems. The literature on frieze patterns grew vastly and continues to expand.

In this note we give a brief overview of some of the directions emerging from~\cite{CC} and single out core ideas which were later replayed in different contexts.

\section{Conway -- Coxeter's frieze patterns: tables of numbers and triangulated polygons}
\label{sec-comb}

\subsection{What is a frieze pattern?}

A {\em frieze pattern of width $n$} is a grid of real numbers as in Fig.~\ref{fr-ex} with $n+2$ rows\footnote{In the original Coxeter's definition a frieze pattern of width $n$ is called a {\em frieze pattern of order $n+3$}. Also, Coxeter's frieze patterns have two additional rows of zeros, one above and one below, which we omit.}, where the first and last rows
consist of $1$'s and the four numbers in every small diamond $a$\put(3,5){$b$}\put(3,-6){$c$} \quad  $d$  satisfy the  {\it diamond
  rule}: $ad - bc = 1$.

%\smallskip

We are mostly interested in {\em  positive integral} frieze patterns, i.e. ones with all entries being positive integers, see Fig.~\ref{fr-ex} for an example.

\begin{figure}[!h]
\begin{center}
  \epsfig{file=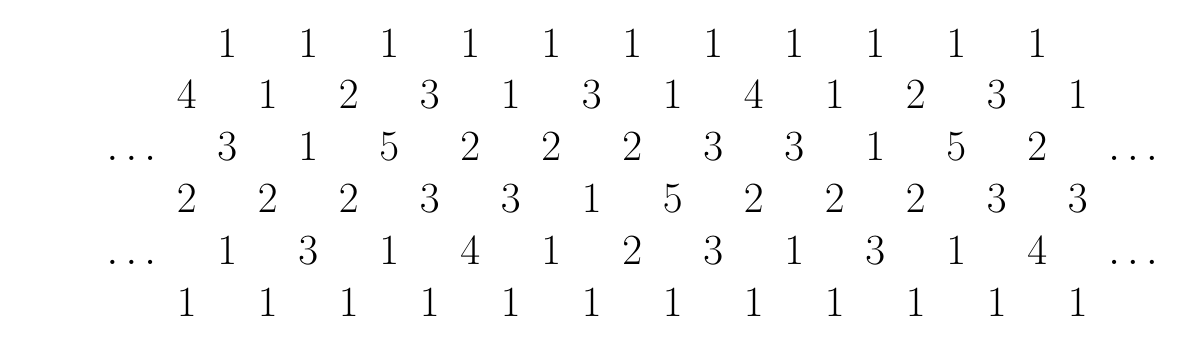,width=0.69\linewidth}
  \caption{Frieze pattern of width $4$.}
\label{fr-ex}
\end{center}
\end{figure}

Frieze patterns satisfy a number of nice properties, some of which can be seen in Fig.~\ref{fr-ex}, for example periodicity and glide symmetry (which is the reason for the name).  Many of the properties were observed already in~\cite{C},  where they were proved based on continued fractions, determinantal identities and difference equations. We will follow Conway and Coxeter~\cite{CC} and discuss frieze patterns in terms of triangulated polygons.

\subsection{How to build a positive integral frieze pattern?}
\label{build}

A positive integral frieze pattern of width $n$ can be constructed from a triangulation of a convex $(n+3)$-gon.   We present two ways to construct a positive integral frieze pattern, both based on the work of Conway and Coxeter~\cite{CC} and leading to the same result.

\subsubsection{Building a frieze pattern from a triangulation: quiddity sequence}
\label{quid}

Consider a triangulation $T$ of an $(n+3)$-gon $P_0\dots P_{n+2}$. Denote by $\alpha_i$ the number of triangles of $T$ containing vertex $P_i$, and extend the sequence $\alpha_0,\dots,\alpha_{n+2}$ periodically to a bi-infinite sequence by defining $\alpha_i=\alpha_{i+(n+3)}$. Now, put the sequence $(\alpha_i)$ as the first non-trivial row of a frieze pattern (this row is also called {\em quiddity sequence}). Due to the diamond rule, the second row is uniquely defined, and applying the diamond rule iteratively we can compute all further rows. 

It was proved by Conway and Coxeter~\cite[Problems (27),(28)]{CC} that the procedure above results in a  positive integral frieze pattern of width $n$ (we will explain the reasons for that in the next section). Moreover, all  positive integral frieze patterns can be constructed in this way (see Theorem~\ref{thm-CC}).  

\begin{figure}[!h]
\begin{center}
  \epsfig{file=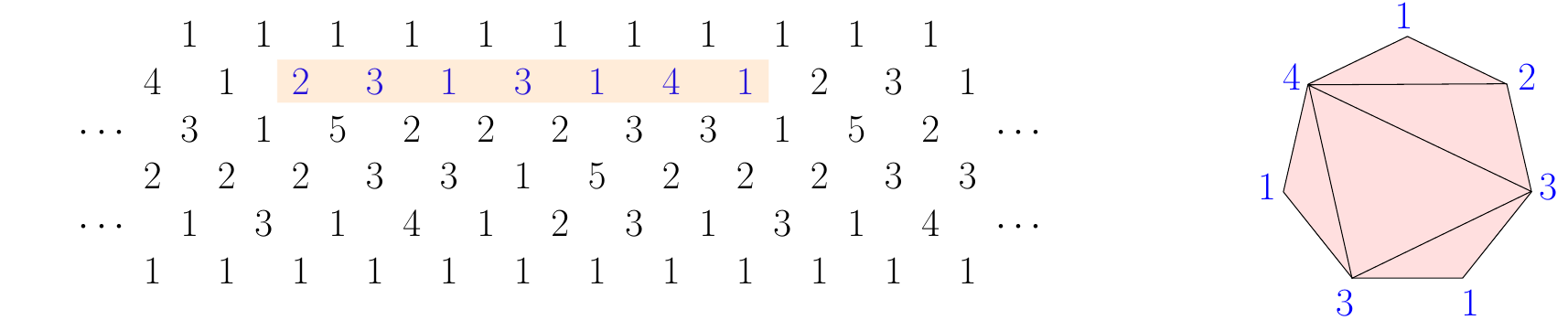,width=0.89\linewidth}
  \caption{Frieze pattern of width $4$ and the corresponding triangulated polygon. The period of the quiddity sequence is highlighted.}
\label{fr-ex-tr}
\end{center}
\end{figure}

\begin{remark}
  \label{match}
A combinatorial interpretation of the entries in the first row of the frieze pattern constructed above was extended to all entries in terms of matchings of vertices and triangles by Broline, Crowe and Isaacs, see~\cite[Theorem~1]{BCI}.  

\end{remark}

\subsubsection{Building a frieze pattern from a triangulation: Ptolemy relations}
\label{Pt}

In this section, we will first look at the polynomial relations between the lengths of sides and diagonals in a Euclidean polygon inscribed in a circle. We then use these relations to build a positive integral frieze pattern.

Consider a cyclic Euclidean $(n+3)$-gon $P$ with vertices $P_0,\dots,P_{n+2}$. Denote $u_{i,j}:=|P_iP_j|$ to be the length of a diagonal or a side with endpoints $P_i$ and $P_j$, where $i,j$ are considered modulo $n+3$, $u_{i,j}=u_{j,i}$. We can now arrange the numbers $u_{i,j}$ in a table as in Fig~\ref{uij}. Notice that the first and the last rows consist of the lengths of sides of $P$.

\begin{figure}[!h]
  \epsfig{file=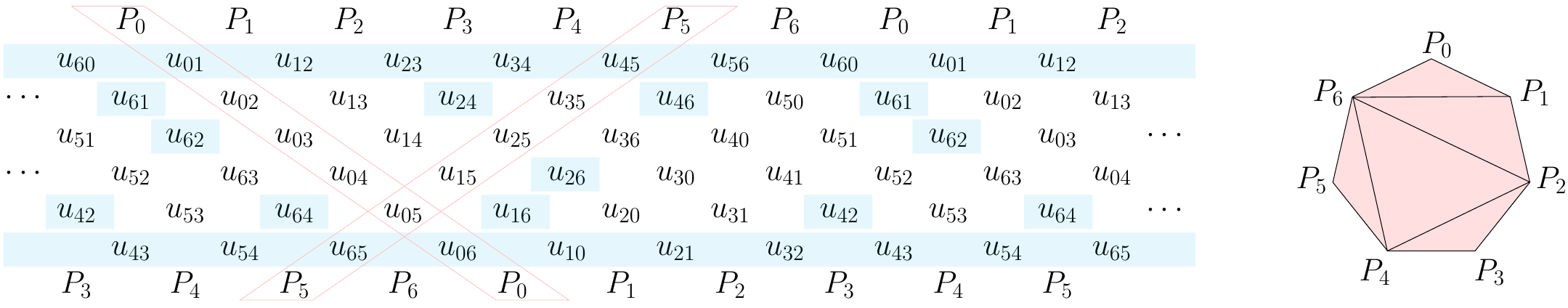,width=0.99\linewidth}  
\caption{Arranging the lengths of diagonals of an $(n+3)$-gon. Entry $u_{i,j}$ is located at the intersection of a line from $P_i$ to $P_i$ in the south-east direction and from $P_j$ to $P_j$ in the north-east direction. The entries corresponding to the sides and diagonals in the triangulation are highlighted.}
\label{uij}
\end{figure}

Observe that every quadrilateral $P_iP_jP_kP_l$ is cyclic and hence, its sides and diagonals satisfy the {\em Ptolemy relation}, see Fig.~\ref{fig-pt}.

\begin{theorem}[Ptolemy]
In a cyclic quadrilateral $ABCD$ the following relation holds: $$|AC|\cdot |BD|=|AB|\cdot|CD|+|BC|\cdot|AD|.$$
  
\end{theorem}

\begin{figure}[!h]
  \epsfig{file=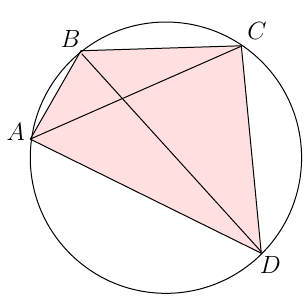,width=0.2\linewidth}  
\caption{Ptolemy relation: $|AC|\cdot |BD|=|AB|\cdot|CD|+|BC|\cdot|AD|$.}
\label{fig-pt}
\end{figure}

Now, choose any triangulation $T$ of $P$, it contains precisely $n$ diagonals. Let $l_1,\dots,l_n$ be the lengths of the diagonals in $T$. % (i.e. $l_k=u_{i,j}$ for some $i,j$). 
Then the length $u_{i,j}$ of every diagonal $P_iP_j$ can be expressed through $l_1,\dots,l_n$ and the lengths of sides $u_{i,i+1}$, $i=0,\dots, n+2$, by applying the Ptolemy relations several times:

\begin{prop}[Laurent phenomenon in a polygon]
  \label{laurent}
If $T$ is a triangulation of $P$,  $l_1,\dots,l_n$ are the lengths of diagonals in the triangulation $T$, then    
 $$u_{i,j}=Q_{i,j}/\prod l_r ,$$
 where the numerator $Q_{i,j}$ is a polynomial with positive integer coefficients in  $l_1,\dots,l_n$ and $u_{k,k+1}$, $k=0,\dots, n+2$,  and the denominator is the product of all lengths $l_r$ of the diagonals in $T$ intersected by $P_iP_j$.

\end{prop} 
In other words, all $u_{i,j}$ can be expressed as Laurent polynomials in $l_1,\dots,l_n$ and $u_{k,k+1}$.

%\bigskip
%\noindent
\subsubsection*{Idea of proof:}
The proof is by induction on the number of intersections of $P_iP_j$ with diagonals belonging to $T$. For diagonals of $P$ belonging to $T$ we have $u_{i,j}=l_k$ for some $k$. Now, assume  we constructed $Q_{u,v}$ for every diagonal $P_uP_v$ intersected  by at most $m-1$ diagonals lying in $T$. The polynomial $Q_{i,j}$ for a diagonal $P_iP_j$ intersecting $m$ diagonals of $T$ can be found as follows. If $P_sP_t$ is any diagonal of $T$ intersected by $P_iP_j$, then for all sides in the quadrilateral $P_iP_sP_jP_t$ the corresponding polynomials are already known, and applying the Ptolemy relation in this quadrilateral we obtain the required expression for $P_iP_j$.
\qed

%\medskip

\begin{remark}
  \label{well}
  The Laurent polynomials  $u_{i,j}=Q_{i,j}/\prod l_r$ can be computed explicitly, and their coefficients can be interpreted combinatorially, see~\cite{S}. Moreover,  all $u_{i,j}$ (as Laurent polynomials) satisfy Ptolemy relations. 
  
\end{remark}

\begin{remark}
  \label{cluster-l}
  Proposition~\ref{laurent} is a very particular case of the Laurent phenomenon for cluster algebras discovered by Fomin and Zelevinsky~\cite[Theorem~3.1]{FZ}, applied to cluster algebras of finite type $A_n$.   
  \end{remark}

We are now ready to construct a positive integral frieze pattern by specializing the values of the Laurent polynomials $u_{i,j}$ in Fig.~\ref{uij}.

\begin{itemize}
  \item
First, we set all $u_{k,k+1}$ to be equal to $1$. In this way, we see that every entry in the first and the last rows of the table is equal to $1$.

\item
Next, we set all $l_k=1$. Since all the denominators of all Laurent polynomials $u_{i,j}$ consist of products of $l_k$, and the polynomials in the numerators have positive integer coefficients, we see that the table is now filled with positive integers.

\item
Finally, consider a quadrilateral $P_iP_{i+1}P_jP_{j+1}$ for $|i-j|>1$. The corresponding Ptolemy relation $u_{i,j}u_{i+1,j+1}=u_{i+1,j}u_{j+1,i}+u_{i,i+1}u_{j,j+1}$ can be written as $u_{i,j}u_{i+1,j+1}-u_{i+1,j}u_{j+1,i}=1$, which is exactly the diamond rule for the diamond
$$\begin{matrix}
  &{u_{i+1,j}}&\\
  u_{i,j}&&{u_{i+1,j+1}}\\
  &{u_{i,j+1}}&
\end{matrix}$$
in the table. 
\end{itemize}

Summarizing, we constructed a table of positive integers satisfying the definition of a frieze pattern.

\begin{remark}
\label{no-Eucl}  
There is no Euclidean triangulated $(n+3)$-gon for $n>0$ satisfying the assumptions $u_{i,i+1}=1$ and $l_k=1$ for all $i$ and $k$ simultaneously. However, one can use hyperbolic geometry to give the values $u_{i,j}$ geometric meaning, see Section~\ref{sec-hyp}.

\end{remark}

\begin{remark}
  \label{quid-proof}
  It can be easily checked from the inductive reasoning above that the entry $u_{i,i+2}$ of the obtained frieze pattern, corresponding to the short diagonals, is equal to the number of triangles containing the vertex $P_{i+1}$. Thus, the entries $u_{i,i+2}$ compose a sequence identical to the quiddity sequence from the previous section (up to shift). This explains the construction in Section~\ref{quid}: it gives rise to the very same frieze pattern we constructed in the present section.

  \end{remark}

\subsection{Conway -- Coxeter's classification of positive integral frieze patterns}
\label{class}

In Section~\ref{build} we provided a construction of a positive integral frieze pattern of width $n$ from a triangulation of an $(n+3)$-gon. Conway and Coxeter proved that every positive integral frieze pattern can be obtained in the way described above. 

\begin{theorem}[\cite{CC}, Problems (28),(29)]
\label{thm-CC}  
Positive integral frieze patterns of width $n$ are in bijection with triangulations of an $(n+3)$-gon. 

More precisely,  for any positive integral frieze pattern of width $n$ there exists a unique triangulation $T$ of an $(n+3)$-gon $P_0\dots P_{n+2}$ such that $u_{i,j}=1$ whenever $P_iP_j$ belongs to $T$, where entries $u_{i,j}$ of the frieze pattern are labeled as in Fig.~\ref{uij}.

\end{theorem}

In view of the construction in Section~\ref{Pt}, Theorem~\ref{thm-CC} can be understood as follows: every positive integral frieze pattern can be constructed from a triangulated polygon by assigning $1$ to every side and diagonal in the triangulation, and then applying Ptolemy relations.

Theorem~\ref{thm-CC} implies that every frieze pattern of finite width $n$ is {\em periodic}:  it is invariant under a horizontal shift with period $n+3$. Moreover, it is invariant under a glide reflection (whose square is the horizontal shift above).
  
 \begin{remark}
    \label{coeff}
    One can modify the definition of a frieze pattern to allow the entries $u_{i,i+1}$ to be arbitrary real numbers, these are called {\em frieze patterns with coefficients}. To see an example of an integral frieze pattern with coefficients, one can take any subpolygon of a polygon corresponding to a Conway -- Coxeter's frieze pattern with some sides not belonging to the initial triangulation.

    Cuntz, Holm and Jorgensen~\cite[Theorem~2.11]{CHJ} proved that every {\em tame} positive integral frieze pattern with coefficients of width $n$ is glide-symmetric and $(n+3)$-periodic, its entries satisfy Ptolemy relations, and thus every such frieze pattern  can be obtained by specializing the Laurent polynomials from Prop.~\ref{laurent} (where tameness condition is defined as in Definition~\ref{sl-def}). In other words, all such frieze patterns arise from assigning positive integers to sides and diagonals composing a triangulation of an $(n+3)$-gon and applying Ptolemy relations to recover all other entries.  It was also shown in~\cite{CHJ,CH} that such frieze pattern arises from a subpolygon of a polygon corresponding to a Conway -- Coxeter's frieze pattern if and only if certain divisibility conditions hold.

    \end{remark}

\section{Frieze patterns and hyperbolic geometry}
\label{sec-hyp}

In Section~\ref{build} it is described how to build a frieze pattern from a triangulated polygon by repeatedly applying Ptolemy relations. As it was observed in Remark~\ref{no-Eucl}, the entries of constructed frieze patterns do not represent Euclidean lengths of the diagonals in a Euclidean polygon. We will now see that the entries actually have a geometric meaning in terms of hyperbolic geometry. For this, we follow Penner~\cite{P} to introduce {\em lambda lengths}.

\subsection{Penner's lambda lengths and the Ptolemy relation}
\label{lambda}

Let $\H^2$ be a hyperbolic plane. One can think of it as of an upper half-plane $\{z\in \C \mid \mathrm{Im}(z)\ge 0\}$ endowed with the metric  
$ds=\frac{\sqrt{dx^2+dy^2}}{y}$ for $z=x+iy$.  One can also consider a model of hyperbolic plane in the unit disc obtained by application of a M\"obius map to the upper half-plane. In both cases, geodesics are represented by arcs of circles or lines orthogonal to the boundary of the disc or the half-plane. 
A {\it horocycle centered at $X\in \partial \H^2$} is a limit of hyperbolic circles passing through a fixed point with the center tending to $X$. In the two models above, a horocycle is represented by a Euclidean circle containing $X$ and tangent to $\partial\H^2$ (or, if $X=\infty$ in the upper half-plane, a horocycle centered at $X$ is represented by a horizontal line ${\mathrm{Im}}(z)=c$).

\begin{figure}[!h]
  \epsfig{file=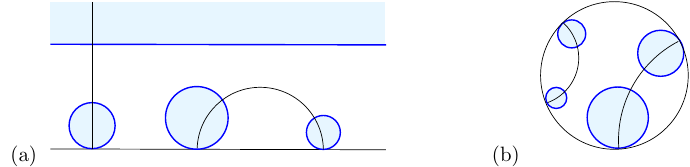,width=0.69\linewidth}  
\caption{Lines and horocycles in the hyperbolic plane:  (a) upper half-plane model; (b) Poincar\'e disc. }
\label{h}
\end{figure}

Although the distance between any two points lying at the boundary of $\H^2$ is infinite, we can measure the distance if we choose some horocycles in advance. Namely, given two points $A,B\in\partial\H^2$, choose a horocycle $h_A$ centered at $A$ and a horocycle $h_B$ centered at $B$. Consider the intersections $A'=AB\cap h_A$ and $B'=AB\cap h_B$ of the hyperbolic line $AB$ with $h_A$ and $h_B$, and let $d$ be the {\em signed distance} between $A'$ and $B'$: 
$$
d=d_{AB}=\begin{cases}
  \text{ the hyperbolic length } |A'B'| \text{ if } h_A\cap h_B=\emptyset;\\
  -|A'B'| \text{ otherwise.}
\end{cases}
$$
The {\it lambda length} $\lambda_{AB}$ (of $AB$ with respect to the horocycles $h_A$ and $h_B$) is defined by
$$
\lambda_{AB}=e^{\frac{d}2}.
$$
By construction, lambda lengths are positive numbers, with $\lambda_{AB}=1$ if and only if the horocycles $ h_A$ and $h_B$ are tangent.

\begin{figure}[!h]
  \epsfig{file=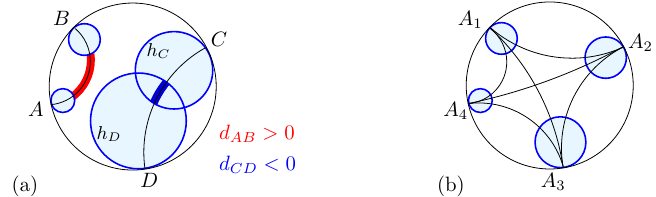,width=0.69\linewidth}  
\caption{ (a) The signed distance $d$ used to define lambda lengths; (b) the hyperbolic Ptolemy relation $\lambda_{13}\lambda_{24}=\lambda_{12}\lambda_{34}+\lambda_{23}\lambda_{14}$, where $\lambda_{ij}=\lambda_{A_iA_j}$. }
\label{fig-lambda}
\end{figure}

It was shown by Penner that lambda lengths also satisfy the Ptolemy relation.

\begin{theorem}[\cite{P}, Proposition~2.6]
  \label{Pt-hyp}
Let $A_1A_2A_3A_4$ be a hyperbolic quadrilateral with $A_1,\dots,A_4\in \partial \H^2$, and let $h_i$ be a horocycle at $A_i$, $i=1,\dots 4$. Let $\lambda_{ij}=\lambda_{A_iA_j}$, $1\le i,j\le 4$ be the lambda lengths with respect to the horocycles $h_i,h_j$.  Then  
$$
\lambda_{13}\lambda_{24}=\lambda_{12}\lambda_{34}+\lambda_{23}\lambda_{14}.
$$

\end{theorem}  
Note that the hyperbolic Ptolemy relation holds for any choice of horocycles.

A hyperbolic polygon with all vertices at $\partial\H^2$ is called a {\em ideal polygon}. An ideal polygon with a fixed horocycle at every vertex is called a {\em decorated ideal polygon}. Theorem~\ref{Pt-hyp} can be interpreted as follows: for a decorated ideal polygon, the lambda lengths of its sides and diagonals satisfy the same Ptolemy relations as Euclidean lengths of a cyclic Euclidean polygon.

\subsection{Entries of frieze patterns as lambda lengths}
\label{frieze-hyp}

In Section~\ref{sec-comb}, we saw that every positive integral frieze pattern can be constructed from a triangulated polygon by assigning $1$ to every side and diagonal in the triangulation, and then applying Ptolemy relations.
The following proposition shows that every triangulated polygon can be realized as a decorated ideal hyperbolic polygon with all lambda lengths of sides and diagonals in the triangulation equal to $1$.

\begin{prop}
\label{poly1}
Let $P=P_0\dots P_{n+2}$ be a triangulated polygon. There exists  an ideal decorated hyperbolic polygon $A=A_0\dots A_{n+2}$ triangulated in the same way as $P$, such that all lambda lengths $\lambda_{A_iA_j}$ are equal to $1$.

  \end{prop}

\begin{figure}[!h]
  \epsfig{file=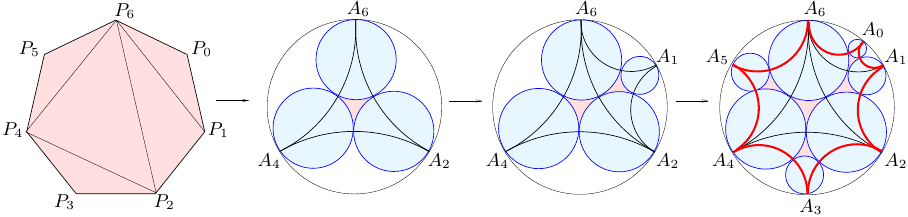,width=0.99\linewidth}  
\caption{Constructing a decorated ideal hyperbolic polygon from a combinatorial triangulated polygon. }
\label{replace}
\end{figure}

\begin{proof}
  The idea of the proof is shown in Fig.~\ref{replace}. We start with an ideal triangle with three mutually tangent horocycles, call its vertices $A_i,A_j,A_k$ (assuming a triangle  $P_iP_jP_k$ belongs to the triangulation $T$ of $P$).  By construction, the lambda lengths of the sides of the decorated ideal triangle  $A_iA_jA_k$ are equal to $1$. Now, reflect the triangle $A_iA_jA_k$ with respect to $A_iA_j$ if $P_iP_j$ belongs to $T$, and denote the new vertex by $A_m$, where $P_iP_j$ is contained in a triangle $P_iP_jP_m$ in $T$. We obtained a triangulated decorated ideal quadrilateral $A_iA_kA_jA_m$ in which, by construction, lambda lengths of all sides and diagonals in the triangulation are equal to $1$.

  In other words, we replace triangles of the triangulation of $P$ with decorated ideal hyperbolic triangles. Continuing this procedure by replacing the triangles of $T$ one by one, we obtain the required polygon. 

\end{proof}

We now want to use the results of this section and Conway -- Coxeter's classification theorem to interpret the entries of any positive integral frieze pattern as lambda lengths of diagonals  of a decorated ideal hyperbolic polygon.

Starting with a positive integral frieze pattern,  Theorem~\ref{thm-CC} provides a unique  triangulated polygon $P$. Proposition~\ref{poly1} then constructs a decorated ideal hyperbolic polygon $A$. Now, observe that lambda lengths of sides and diagonals of $A$ and elements of the frieze pattern satisfy the same rules: the ones corresponding to the sides and the diagonals of the triangulation are equal to $1$, and they all satisfy Ptolemy relations. From this, we conclude that the entry $u_{ij}$ of the initial frieze pattern is precisely the lambda length $\lambda_{A_iA_j}$. In particular, the quiddity row formed by entries $u_{i,i+2}$ consists of lambda lengths $\lambda_{A_iA_{i+2}}$ of short diagonals of the polygon $A$.

\begin{remark}
  \label{in-farey} 
The construction above also follows from  results of Morier-Genoud, Ovsienko and Tabachnikov~\cite[Theorem~1]{MGOT} who proved that positive integral frieze patterns are in bijection with polygons embedded in the Farey graph (up to isometries of the hyperbolic plane). Further development of ideas from~\cite{MGOT} can be found in Section~\ref{sec-sh}. 

  \end{remark}

  \section{$SL_2$-tilings and the Farey graph}
  \label{sec-Farey}
  
  In the remark to their problem (16), Conway and Coxeter suggested to consider {\em infinitely tall}  frieze patterns. Generalization of this brings us to the notion of $SL_2$-tiling
introduced by Assem, Reutenauer and Smith~\cite{ARS}.  

\subsection{$SL_2$-tilings}

\begin{definition}[$SL_2$-tiling]
  \label{sl-def}
An {\it $SL_2$-tiling} is an infinite matrix $(a_{ij})$, $i,j\in \Z$, such that all adjacent $2\times 2$ minors  are equal to $1$. 
An $SL_2$-tiling is {\it tame} if all $3\times 3$ adjacent minors vanish. 

\end{definition}

\begin{example}
  \label{ex-sl2}
Problem (11) of Conway and Coxeter shows that if the positivity assumption is dropped, then any frieze pattern can be extended by two rows of zeroes and negative $1$'s above and below. This gives rise to an elementary example of an  $SL_2$-tiling constructed by ``gluing'' copies of frieze pattern with alternating signs (rotated by $\pi/4$ clockwise), see Fig.~\ref{fig-sl2}(a,b). Another example can be obtained from gluing together the infinite frieze pattern from Problem (16) and its negative, see  Fig.~\ref{fig-sl2}(c,d). 
  
  \end{example}

\begin{figure}[!h]
  \epsfig{file=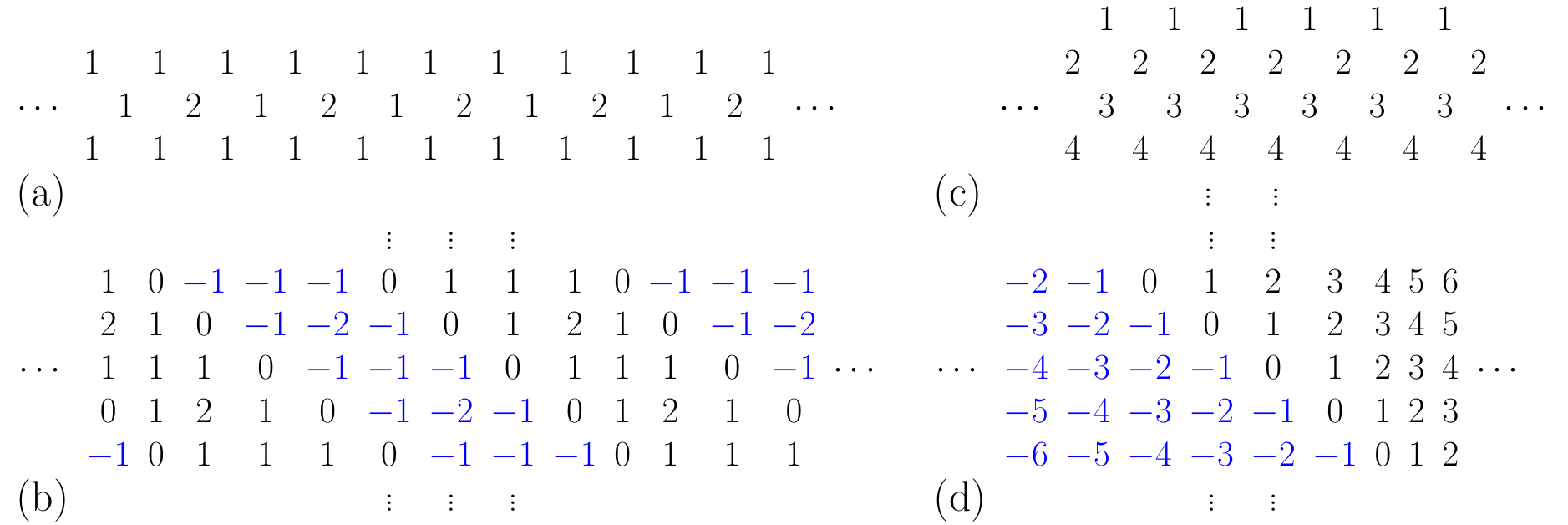,width=1.0\linewidth}  
\caption{Two examples of $SL_2$-tilings. (b) $SL_2$-tiling obtained by gluing of infinitely many copies of a frieze pattern of width $1$ and its negative (cf.~\cite[(11)]{CC}); (d)  $SL_2$-tiling obtained by gluing of an ``infinitely tall'' frieze pattern and its negative (cf.~\cite[(16)]{CC}).}
\label{fig-sl2}
\end{figure}

Notice that both $SL_2$-tilings in the example above are tame. A classification of tame $SL_2$-tilings (following ideas of Bergeron and Reutenauer~\cite{BR} and of Morier-Genoud, Ovsienko and Tabachnikov~\cite{MGOT}, and generalizing earlier results of Bessenrodt, Holm and Jorgensen~\cite{BHJ2}) was given by Short~\cite{Sh} in terms of paths in the {\em Farey graph}.

\subsection{The Farey graph and tame $SL_2$-tilings}
\label{sec-sh}

 \begin{definition}[Farey graph]
The  {\it Farey graph} $\mathcal F$ is the graph with vertices at $\Q\cup\{ \infty\}$, where the two vertices given by reduced fractions $p/q$ and $r/s$ are connected by an edge whenever $|ps-qr|=1$. 
   
 \end{definition}

It is convenient to depict the Farey graph in the upper half-plane model of the hyperbolic plane, drawing the edges as geodesics connecting the corresponding points of the boundary. Then the edges cut the hyperbolic plane into ideal triangles (where two adjacent triangles are obtained from each other by a reflection), see Fig.~\ref{fig-farey}(a).

\begin{figure}[!h]
  \epsfig{file=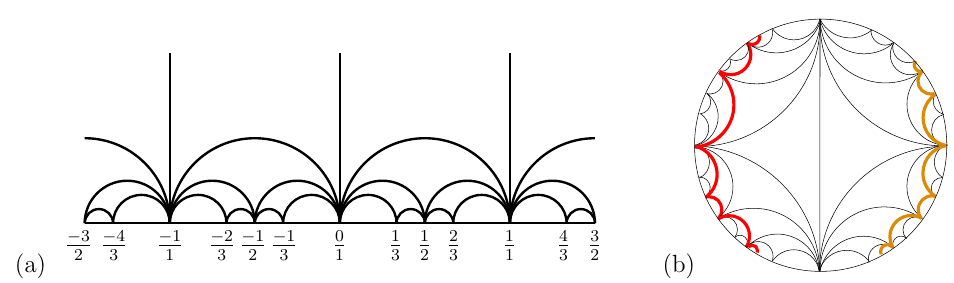,width=0.99\linewidth}  
\caption{ (a) Farey graph in the upper half-plane; (b) same in the disc, with two paths highlighted. }
\label{fig-farey}
\end{figure}

The Farey graph on the upper half-plane can be supplied with the set of horocycles given by {\em Ford circles}: these are circles of diameter $1/q^2$ at every rational point $p/q$ (written in reduced form), see Fig.~\ref{fig-ford}(a). Then one can check that the lambda length between $p/q$ and $r/s$ with respect to Ford circles is given by $$\lambda_{\frac{p}q,\frac{r}s}=|ps-qr|.$$ In particular, the lambda lengths of edges of $\mathcal F$ are equal to $1$, and lambda lengths between any two rational points are integer.

A bi-infinite sequence of reduced fractions $(a_r/b_r)_{r\in \Z}$ represents a {\em path} in the Farey graph $\mathcal F$ if $a_rb_{r+1}-b_ra_{r+1}=1$ for all $r$.
We consider two bi-infinite paths $(a_r/b_r)$ and $(c_r/d_r)$ and define the map  $M: \Z\times\Z \to \Z$ by
 \begin{equation}
 \label{m}  
 m_{i,j}=a_id_j-b_ic_j.
\end{equation}

It turns out that this map provides a bijection between tame $SL_2$-tilings and pairs of paths in $\mathcal F$.

\begin{theorem}[\cite{Sh}, Theorem 1.1]\label{SL2-Farey}
The table $\{m_{i,j} \}$  defined in {\em (\ref{m})} constitutes a tame $SL_2$-tiling. Every tame $SL_2$-tiling can be obtained from a pair of bi-infinite paths as in {\em (\ref{m})} defined uniquely (up to simultaneous action of $SL_2(\Z)$). 

\end{theorem}  

As can be seen from Example~\ref{ex-sl2} (see Fig.~\ref{fig-sl2}(a,b)), every frieze pattern gives rise to an $SL_2$-tiling. Theorem~\ref{SL2-Farey} provides a way to recover Conway -- Coxeter's classification of frieze patterns of width $n$: one needs to replace the  pair of paths by two copies of the same closed non-self-intersecting path of length $n+3$. This path encloses a triangulated polygon from Remark~\ref{in-farey}, and the quiddity sequence of the frieze pattern is precisely the number of triangles in the triangulation counted at every vertex.

\begin{figure}[!h]
  \epsfig{file=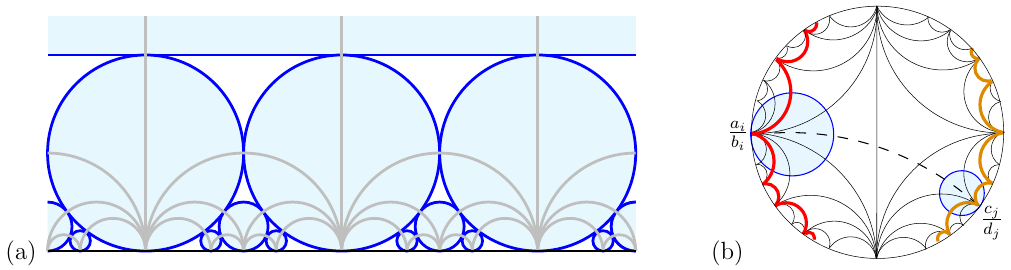,width=0.99\linewidth}  
\caption{ (a) Ford circles as horocycles at the vertices of the Farey graph; (b) construction of an $SL_2$-tiling by using two paths of the Farey graph: $m_{i,j}=|a_id_j-b_ic_j|=\pm\lambda_{\frac{a_i}{b_i},\frac{c_j}{d_j}}$.}
\label{fig-ford}
\end{figure}

\subsection{Infinite frieze patterns}
\label{rem-inf}
A particular case of a tame $SL_2$-tiling is when two paths $(a_r/b_r)$ and $(c_r/d_r)$ coincide, such an  $SL_2$-tiling is then composed of two halves which differ by sign and are separated by a diagonal of zeroes (see Fig.~\ref{fig-sl2}(d) for an example). It is natural to identify a half of it with an {\em infinite frieze pattern}: it contains a row of zeroes, a row of $1$'s, and then satisfies the diamond rule. As before, the first non-trivial row is called the quiddity sequence.

In terms of the path  $(a_r/b_r)$ defining the frieze pattern, the element of the quiddity sequence $\alpha_r$ is the lambda length between  $a_{r-1}/b_{r-1}$ and $a_{r+1}/b_{r+1}$ (up to negative sign). Equivalently, $|\alpha_r|$ is equal to the number of triangles in the Farey graph incident to  $a_{r}/b_{r}$ and lying ``inside'' the path. Theorem~\ref{SL2-Farey} implies that any integer quiddity sequence gives rise to an integer infinite frieze pattern. We will see more on this in Section~\ref{inf}.

\subsection{$SL_k$-tilings}
Bergeron and Reutenauer~\cite{BR} generalised the notion of $SL_2$-tiling as follows. An {\em $SL_k$-tiling} is an infinite matrix of integers with adjacent $k\times k$ minors equal to $1$, with tameness defined by the vanishing of all $(k+1)\times(k+1)$ adjacent minors.

The tameness condition implies that an $SL_k$-tiling is completely defined by its consecutive $k$ rows (or $k$ columns). This was used in~\cite[Proposition~3]{BR} to provide a classification of all tame $SL_k$-tilings: these are in bijection with triples $(A,u_r,v_r)$, where $A\in SL_k(\Z)$ and $u_r$ and $v_r$ are two bi-infinite sequences of vectors in $\Z^{k-1}$ (vectors in $u_r$ and $v_r$ serve as coefficient tuples of linear dependencies).

For $k=2$, $u_r$ and $v_r$ are just bi-infinite sequences of integers which can be interpreted in terms of Theorem~\ref{SL2-Farey} as follows: if an  $SL_2$-tiling is defined by two paths $(a_r/b_r)$ and $(c_r/d_r)$, then $u_r=a_{r-1}b_{r+1}-b_{r-1}a_{r+1}$, and $v_r=c_{r-1}d_{r+1}-d_{r-1}c_{r+1}$. In other words, $|u_r|$ and $|v_r|$ are lambda lengths between the $(r-1)$-st and $(r+1)$-st points of the paths $(a_r/b_r)$ and $(c_r/d_r)$ respectively.
In particular, if an $SL_2$-tiling is constructed from a positive integral frieze pattern (as in Example~\ref{ex-sl2}), both sequences $u_r$ and $v_r$ coincide with the quiddity sequence of the frieze pattern. This shows that the sequences $u_r$ and $v_r$ of vectors in $\Z^{k-1}$ in the classification of $SL_k$-tilings by Bergeron and Reutenauer can also be understood as {\em quiddities} of these tilings.

%\bigskip

By making use of the classification above a counterpart of Theorem~\ref{SL2-Farey} for  $SL_k$-tilings was provided by Peterson and Serhiyenko~\cite{PS}.

A {\it  path} $\gamma=(\gamma_r)_{r\in\Z}$ is a bi-infinite sequence of vectors $\gamma_r\in \Z^k$ with the property that the matrix $(\gamma_r,\dots,\gamma_{r+k-1})$, whose columns are $k$ consecutive entries of $\gamma$, has determinant $1$. Given two paths $\gamma$ and $\delta$,  define a map $M: \Z\times\Z \to \Z$ by
 \begin{equation}
 \label{mm}  
 m_{i,j}=\det(\gamma_i,\dots,\gamma_{i+k-2},\delta_j).
\end{equation}

\begin{theorem}[\cite{PS}, Theorem 1.1]
  The table $\{m_{i,j} \}$  defined in (\ref{mm}) constitutes a tame $SL_k$-tiling. Every tame $SL_k$-tiling can be obtained from a pair of bi-infinite paths as in (\ref{mm}) defined uniquely (up to simultaneous action of $SL_k(\Z)$).

\end{theorem}

\section{p-angulations and other dissections of polygons}
\label{sec-p}

By subdividing a polygon into polygons with various numbers of sides, one can construct Conway -- Coxeter frieze patterns over different fields or generalized frieze patterns. In this section, we discuss two constructions of this type, both related to the {\em Conway -- Coxeter counting}. 

\subsection{Conway -- Coxeter counting}

Conway -- Coxeter's problem (32) provides an easy combinatorial way to start with a triangulated polygon and compute every entry in the corresponding Conway  -- Coxeter's frieze pattern. It works as follows.

Let $T$ be a triangulation and $\{u_{i,j}\}$ be a positive integral frieze pattern given in the notation of  Fig~\ref{uij} (where all $u_{i,j}$ corresponding to  sides of $P$ and diagonals $P_iP_j$ lying in $T$ are set to $1$). To compute $u_{i,j}$ for a given $i$ and all $j= 0,\dots,n+2$ we label the vertex $P_i$ of the polygon $P$ with $0$, and we label all vertices connected to $P_i$ in $T$ with $1$. Then, whenever a triangle has just two labeled vertices (with labels $b$ and $c$), we label the third vertex with $b + c$ and  continue until every vertex is labeled. Now, $u_{i,j}$ is precisely the label of the vertex $P_j$, see Fig.~\ref{fig-count} for an example.

\begin{figure}[!h]
  \epsfig{file=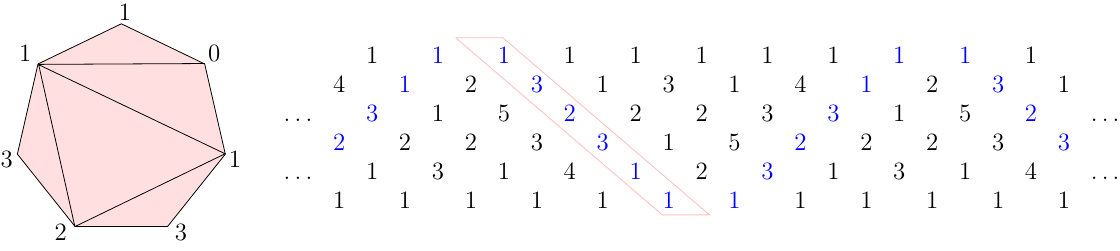,width=0.89\linewidth}  
\caption{Conway -- Coxeter counting. }
\label{fig-count}
\end{figure}

\subsection{$p$-angulations and generalized frieze patterns}

A similar counting was applied by Bessenrodt, Holm and Jorgensen~\cite{BHJ} (see also~\cite{B}) to $p$-angulated polygons (i.e., polygons decomposed into $p$-gons rather than triangles) to construct {\em generalized frieze patterns} of positive integers, where the diamond rule is substituted with $ad-bc$ being equal either to $1$ or to $0$.

More precisely, consider a $p$-angulated $(n+3)$-gon $P$. We are going to associate with it a generalized frieze pattern of width $n$ with entries $u_{i,j}$ corresponding to diagonals $P_iP_j$. To compute $u_{i,j}$ for a given $i$ and all $j= 0,\dots,n+2$, we label the vertex $P_i$ of the polygon $P$ with $0$, and we label all vertices of every $p$-gon containing $P_i$ with $1$. Now, whenever a $p$-gon has just two labeled vertices with labels $b$ and $c$, we label all its other vertices with $b + c$ and  continue until every vertex of $P$ is labeled, see Fig.~\ref{d-ang} for an example. It is shown in~\cite{BHJ} that such generalized frieze patterns inherit many combinatorial properties of Conway -- Coxeter's frieze patterns.

\begin{figure}[!h]
  \epsfig{file=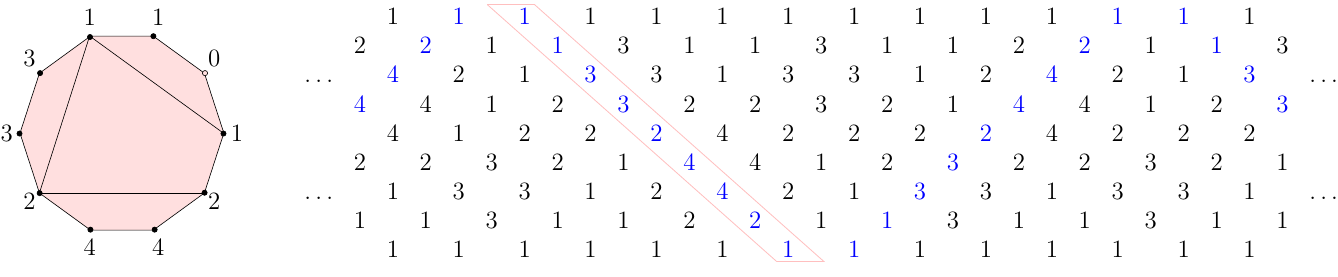,width=0.99\linewidth}  
\caption{Obtaining a generalized frieze pattern from a $p$-angulated polygon by Conway -- Coxeter counting.}
\label{d-ang}
\end{figure}

\subsection{Dissections of polygons and frieze patterns over number fields}

A {\em dissection} of a polygon $P$ is any collection of its mutually non-crossing diagonals. In particular, a triangulation or a $p$-angulation is an example of a dissection. If an $(n+3)$-gon is dissected into $k$ polygons with $p_1,\dots,p_k$ sides respectively, there are several ways (all leading to the same result) to construct a Conway -- Coxeter's frieze pattern of width $n$ with entries belonging to the ring of algebraic integers of the field $K = \Q(w_{p_1}, \dots ,  w_{p_k})$, where $w_p= 2\cos(\pi/p)$.

\subsubsection{Quiddity sequence}
\label{sec-quid}

Let $P$ be an $(n+3)$-gon  dissected into $k$ polygons with $p_1,\dots,p_k$ sides respectively. To every $p_m$-gon in the dissection, assign a {\em weight} $w_{p_m}= 2\cos(\pi/p_m)$. Now, label every vertex $P_i$ of $P$ with the sums of the weights of all polygons in the dissection incident to $P_i$, these labels form a sequence $a_0,\dots,a_{n+2}$, see Fig.~\ref{fig-34} for an example. Defining $a_i=a_{i+(n+3)}$, we obtain a bi-infinite sequence of integers of $K$.

\begin{theorem}[Holm, Jorgensen~\cite{HJ}, Theorem B]
  The bi-infinite sequence $(a_i)$ is a quiddity row of a frieze pattern of width $n$ over the ring of algebraic integers of $K$.
\end{theorem}
It was also shown in~\cite{HJ} that distinct bi-infinite sequences give rise to distinct frieze patterns.

  It is easy to see that if all $p_i=3$, then this procedure recovers Conway -- Coxeter's construction of positive integral frieze patterns described in Section~\ref{quid}. If all $p_i=p$, then the quiddity row is a positive integer sequence multiplied by $w_p$, and it was shown in~\cite[Theorem~A]{HJ} that frieze patterns with such quiddity rows are in bijection with $p$-angulations. 

\begin{figure}[!h]
  \epsfig{file=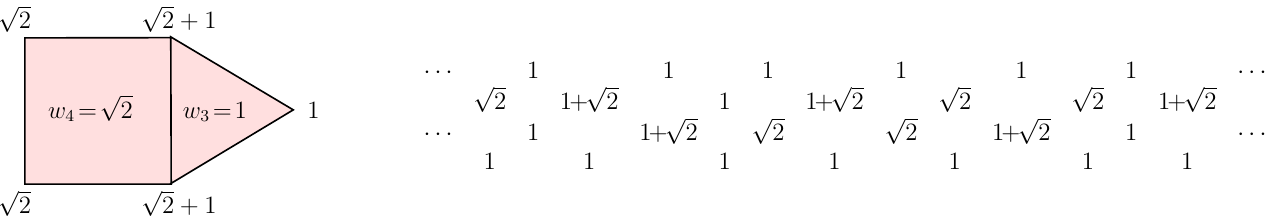,width=0.99\linewidth}  
\caption{Dissection of a pentagon and the corresponding frieze pattern.}
\label{fig-34}
\end{figure}

In particular, if the dissection consists of a unique $p$-gon, one recovers frieze patterns from Conway -- Coxeter's problems (13)--(16).  In other words, entries of such frieze pattern are lengths of diagonals of a regular Euclidean $p$-gon with sides of length $1$.

\begin{remark}
\label{rings}
The ideas above were developed further by Cuntz~\cite{Cu} to study frieze patterns with entries in an arbitrary subset of a commutative ring. The main tool is a decomposition of a frieze pattern into {\em irreducible} ones, where a decomposition corresponds to a dissection of a polygon.  

\end{remark}

\subsubsection{Lambda lengths}
\label{sec-l}

The construction by Holm and Jorgensen described in the previous section can be also interpreted in terms of hyperbolic geometry.

Given a $p$-gon in the dissection of a polygon $P$, we can construct a decorated regular ideal hyperbolic polygon with the same number of sides (here, by a {\em regular $p$-gon} we mean a $p$-gon with maximal symmetry group, i.e., dihedral group $D_{2p}$). There is a unique way to choose horocycles as follows: horocycles at neighboring vertices are tangent, and all the horocycles are invariant with respect to the symmetry group of the $p$-gon. In particular, the lambda lengths of diagonals of this $p$-gon are precisely the lengths of corresponding diagonals of a regular Euclidean $p$-gon with side of length $1$ (this can be easily computed, or one can use the results of Penner, see Remark~\ref{coord}).  

Now, given a dissected polygon $P$, substitute each of the $k$ polygons in the dissection with the corresponding  decorated ideal polygon. This procedure turns $P$ into into a decorated ideal hyperbolic polygon defined uniquely up to the group of hyperbolic isometries. Applying appropriate Ptolemy relations, it is easy to see that the lambda lengths of short diagonals of $P$ are equal to the corresponding entries of the quiddity sequence $(a_i)$ defined in the previous section. Therefore, as discussed in Section~\ref{frieze-hyp}, the lambda lengths of all the diagonals are precisely the entries of the frieze pattern constructed in Section~\ref{sec-quid}.       

\begin{remark}
\label{p-farey}
 It was shown in~\cite{HJ} that if all polygons in the dissection are $p$-gons, then the corresponding frieze pattern can be obtained from a closed path $P_0P_1\dots,P_{n+2}$ in the $p$-counterpart of the Farey graph, which is obtained by reflections of a regular ideal $p$-gon in the hyperbolic plane, see Fig.\ref{quad} for an example in the case $p=4$.

  \end{remark}
\begin{figure}[!h]
  \epsfig{file=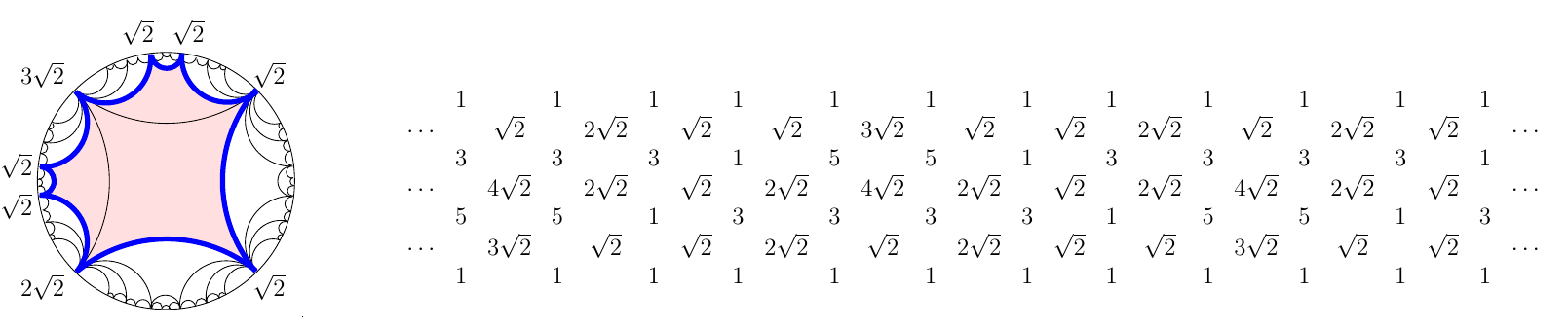,width=0.99\linewidth}  
\caption{A graph obtained by reflections of regular ideal quadrilateral in $\H^2$, a closed path on the graph, and the corresponding frieze pattern.  }
\label{quad}
\end{figure}

  \begin{remark}
    \label{tall}

Following Conway -- Coxeter's remark to problem (16), we can consider the ``limit'' of the constructions above as $p$ tends to infinity, which results in the infinite frieze pattern shown in Fig.~\ref{fig-sl2}(c). This infinite frieze pattern can be obtained by computing lambda lengths between points of an infinite path $\{... ,-k,...,-1,0,1,...,k,....\}$ on a classical Farey graph with respect to Ford circles.

\end{remark}

\subsubsection{Weighted Conway -- Coxeter counting}

As was observed in problem (16) by Conway and Coxeter, the lengths of diagonals of a regular Euclidean $p$-gon with side of length $1$ are equal to $\frac{\sin(m\pi/p)}{\sin(\pi/p)}$ for $m=2,\dots,p-2$. These can be expressed via {\em Chebyshev polynomials of the second kind} defined by $$U_0(x)=1,\quad U_1(x)=2x,\quad U_m(x)=2xU_{m-1}(x)-U_{m-2}(x).$$
Indeed, the length of a short diagonal is equal to $2\cos(\pi/p)$, and for $x=\cos(\pi/p)$ the recurrence relation defining $U_m$ is precisely the Ptolemy relation on diagonals in the $p$-gon. Thus, if a $p$-gon has vertices $P_0,\dots,P_{p-1}$, then the length of $P_0P_m$ is equal to $\frac{\sin(m\pi/p)}{\sin(\pi/p)}=U_{m-1}(\cos(\pi/p))$. 

Now, let $P$ be a dissected polygon. We want to compute the entries $u_{i,j}$ of the frieze pattern for given $i$ and all $j$ corresponding to diagonals $P_iP_j$. As before, assign weight $w_p$ to every $p$-gon in the dissection. For every vertex $P_j$ sharing a $p_k$-gon with $P_i$, label it with $U_m(\cos(\pi/p_k))$ according to the position of $P_j$ in the $p_k$-gon as above. If there are two labeled vertices in a $p_k$-gon  only (with weights $b$ and $c$), then we label the other vertices of this $p_k$-gon as follows: index the vertices $B_0,\dots,B_{p_k-1}$ with $B_0$ labeled $b$ and $B_{p_k-1}$ labeled $c$, and label $B_m$ with $$b\,U_{m-1}(\cos(\pi/p_k))+c\,U_m(\cos(\pi/p_k)).$$   In this way, we can label all vertices of $P$ with the values of $u_{i,j}$.  Note that if the dissection is a triangulation then the procedure above specialises to Conway -- Coxeter counting.

\section{Friezes from surfaces}
\label{surfaces}

The notion of Conway -- Coxeter's frieze pattern can be generalized via substituting triangulated polygons with triangulated marked surfaces.

\subsection{What is a frieze?}
\label{def}

Consider a surface $S$ with or without boundary, and let $M$ be a finite set of marked points on $S$. The marked points may belong to $\partial S$ (we require that every boundary component contains at least one marked point) or to the interior of $S$ (in this case they are called {\em punctures}). We call $(S,M)$ a {\em marked surface}. 

We now want to define equivalence classes of curves with endpoints in $M$. The curves we are interested in are non-self-intersecting, do not pass through other points of $M$, and are considered up to homotopy (relative to $M$). The equivalence classes are called {\em arcs}. 

A {\em triangulation} of $(S,M)$ is a collection of arcs subdividing $S$ into triangles with vertices in $M$.  In the sequel, we will only consider marked surfaces admitting triangulations (i.e., the surface should be ``big enough'': we omit e.g. the sphere or the disc with at most two marked points). Under these assumptions, a triangulation can also be considered as a maximal collection of mutually disjoint arcs (see Fig.~\ref{surf}(a) for an example).

\begin{definition}
Given a marked surface $(S,M)$,  an {\it integral positive frieze on $(S,M)$} (or simply a {\em frieze} for short) is an assignment of positive integers to all arcs on $S$, such that all boundary arcs are assigned $1$, and the Ptolemy relation holds in every quadrilateral embedded into $(S,M)$.
\end{definition}

In particular, taking $(S,M)$ to be a polygon, we obtain precisely Conway--Coxeter's characterization of integral positive frieze patterns. A natural question one can ask is whether a counterpart of Theorem~\ref{thm-CC} holds for friezes on $(S,M)$. This can be split into the following two questions.
\begin{itemize}
 
  \item
    Given a frieze on a marked surface $(S,M)$, is there always a triangulation such that all its arcs are assigned $1$? A frieze possessing such a triangulation is called {\em unitary}.
    \item
  How can one describe unitary and non-unitary friezes?
\end{itemize}

We will address these two questions in the next two sections.

\begin{figure}[!h]
  \epsfig{file=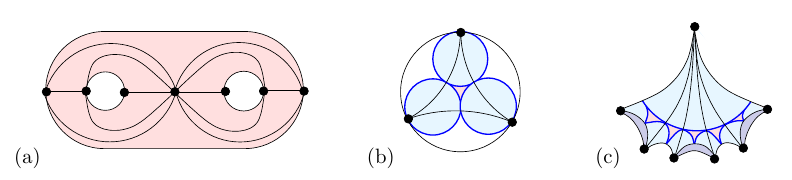,width=0.99\linewidth}  
\caption{ (a) A triangulated marked surface; (b) an ideal triangle with tangent horocycles; (c) sketch of the hyperbolic surface obtained by replacing triangles in (a) by copies of the decorated hyperbolic triangle shown in (b). }
\label{surf}
\end{figure}

\begin{remark}
\label{cluster}
 The definition of a frieze from a marked surface can be considered as a partial case of a frieze from a cluster algebra, where a frieze is defined as a homomorphism from the cluster algebra to $\R$, see~\cite{ARS}. Here, the Ptolemy relations in friezes are precisely exchange relations in cluster algebras. 
  In these settings, while dealing with surfaces with punctures, one needs to consider slightly altered set of arcs (called {\em tagged arcs}) and modified Ptolemy relations (see~\cite{FT}). 
  
  \end{remark}

\subsection{Building unitary friezes}
\label{unit-f}

Actually, unitary friezes are in a one-to-one correspondence with triangulations.
To get an idea about the reasons for that, take a triangulated marked surface, and replace every triangle with a decorated ideal hyperbolic triangle. More precisely, we supply each triangle with (parts of) three mutually tangent horocycles, so they are all copies of the decorated triangle shown in Fig.~\ref{surf}(b).

After this substitution, the surface $S\setminus M$ is endowed with a complete hyperbolic metric of finite area with (a part of) a horocycle around every point of $M$ (see Fig.~\ref{surf}(c)). By construction, the lambda length of every arc of the triangulation is equal to $1$. Due to Theorem~\ref{Pt-hyp}, the lambda lengths of all other arcs can be computed using Ptolemy relations.  As in the case of a polygon (see Section~\ref{Pt}), all lambda lengths turn out to be positive integers. Therefore, the lambda lengths of all arcs form a unitary frieze.

Thus, from every triangulation of $(S,M)$ we can construct a unique unitary frieze. This correspondence is bijective: due to the integrality condition and the Ptolemy relation, arcs of lambda length $1$ cannot intersect each other, so a unitary frieze corresponds to one triangulation only.

\begin{remark}
  \label{coord}
The construction above is an application of a general result by Penner~\cite[Theorem~3.1]{P}: given any triangulation of $(S,M)$ with positive numbers assigned to its arcs, there exists a unique {\em decorated hyperbolic metric} on $(S,M)$ (i.e., a complete hyperbolic metric of finite area on $S\setminus M$ with (a part of) a horocycle around every point of $M$), such that the arcs are geodesic, and the assigned numbers are precisely the lambda lengths.

  \end{remark}

\begin{example}[Triangle with given lambda lengths]
  Given an ideal triangle $A_1A_2A_3$ in $\H^2$ (all ideal triangles are isometric to each other, so we may assume $A_3=\infty$), suppose we want to construct the horocycles $h_1,h_2, h_3$ such that $\lambda_{A_i,A_j}=x_{ij}$, where $x_{ij}\in \R_+$ are given. Consider   a horocycle $h_3$ centered at $\infty$ and given by ${\mathrm{Im\, }} z=t$. Then there exist a unique horocycle $h_1$ centered at $A_1$ such that $\lambda_{A_1,\infty}=x_{13}$ and a unique horocycle $h_2$ centered at $A_2$ such that $\lambda_{A_2,\infty}=x_{23}$, see Fig~\ref{triangle}. We can now measure the lambda length $\lambda_{A_1,A_2}$ as a function of $t$. Notice that $\lambda_{A_1,A_2}(t)$ decreases monotonically from $\infty$ to $0$ as $t$ grows from $0$ to $\infty$, which means that there exists a unique choice of $t$ resulting in $\lambda_{A_1,A_2}=x_{12}$.
\end{example}

\begin{figure}[!h]
  \epsfig{file=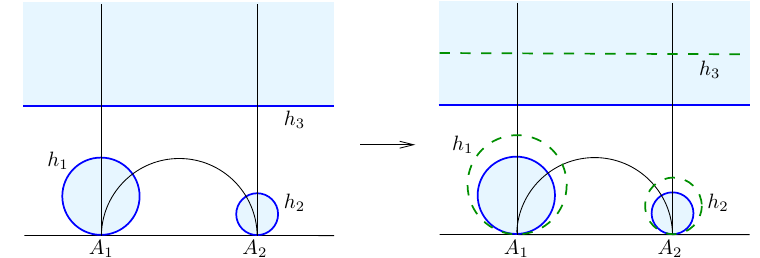,width=0.69\linewidth}  
\caption{Constructing a decorated ideal hyperbolic triangle with prescribed lambda lengths.}
\label{triangle}
\end{figure}

\subsection{Classifying friezes on bordered surfaces}
\label{bordered}

We now look at non-unitary friezes: do they exist, and if yes, can we describe them?

An example of a non-unitary frieze from a once punctured disc  was constructed by Thomas in the appendix to~\cite{BM}, see Fig.~\ref{nonuni}(a). This example was incorporated into a series of examples by Fontaine and Plamondon~\cite{FP}: given a triangulation of a punctured disc, assign to all the arcs incident to the puncture the same value $k$, where $k$ divides the valence of the puncture in the triangulation, and assign $1$ to all other arcs, see Fig.~\ref{nonuni}(b) (in fact, as shown in~\cite{FP}, these exhaust all possible friezes from a punctured disc).

\begin{figure}[!h]
   \epsfig{file=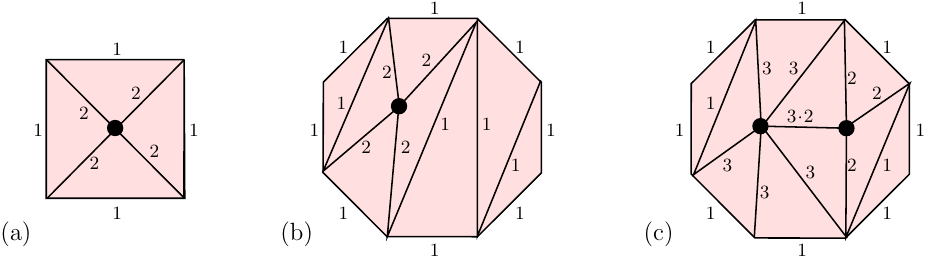,width=0.79\linewidth}
\caption{ Non-unitary friezes on punctured surfaces.}
\label{nonuni}
\end{figure}

The following theorem generalizes the examples of Thomas and Fontaine -- Plamondon, and gives a Conway -- Coxeter style classification of friezes from surfaces with boundary.

\begin{theorem}[\cite{FeTu}]
  \label{unit}
There is a one-to-one correspondence between positive integral friezes from a marked bordered surface $(S,M)$ and triangulations of $(S,M)$ supplied with a collection of positive integers $\{k_i\}$ at all punctures $\{P_i\}$, where each $k_i$ divides the valence of $P_i$ in the triangulation.%, and $k_i^2\le m_i$.

\end{theorem}

\begin{remark}
  Theorem~\ref{unit} implies, in particular, that all friezes from unpunctured surfaces are unitary, and there is a bijective
correspondence between friezes and triangulations (this generalizes earlier results from~\cite{GS,CFGT}).

Also, the number of equivalence classes of friezes from a given bordered surface is finite (where two friezes are equivalent if one can be taken to the other by the mapping class group of $(S,M)$).

\end{remark}

In other words, Theorem~\ref{unit} says that every non-unitary frieze from a surface with boundary can be obtained from an appropriate unitary frieze by multiplying the labels of all arcs incident to  a puncture $P_i$ by an integer $k_i$ for every $i$. See  Fig.~\ref{nonuni}(c) for an example of a non-unitary frieze on a twice punctured disc.

For surfaces without boundary, the same construction leads to non-unitary friezes, but whether they all can be obtained in this way is unknown at the moment.

\begin{remark}
The numbers $k_i$ in the theorem are responsible for the choice of horocycles around the punctures. Namely, taking all $k_i=1$ corresponds to a unitary frieze, while changing $k_i$ substitutes the horocycle around $P_i$ with one lying closer to $P_i$ and thus increases the lambda lengths of all arcs incident to $P_i$. The restrictions on $k_i$ originate from considerations of tagged arcs (cf. Remark~\ref{cluster}).

  \end{remark}

  \subsection{Infinite periodic frieze patterns}
  \label{inf}

  We are now back to considering frieze patterns as tables of numbers. 

  Friezes from surfaces provide an easy way to construct an infinite frieze pattern. Namely, consider a marked surface $(S,M)$ with boundary (different from a polygon), take a triangulation $T$, choose one boundary component, and compute the number of triangles at marked points of this component. This gives a (periodic) quiddity sequence, and thus an infinite (periodic) integral frieze pattern (see Section~\ref{rem-inf}). The entries of this frieze pattern have a geometric meaning: the unitary frieze on $(S,M)$ corresponding to $T$ defines a decorated hyperbolic metric on $(S,M)$, and the entries of the frieze pattern are equal to the lambda lengths of (possibly self-crossing) arcs with endpoints on the boundary component. In particular, the constructed infinite periodic frieze pattern is positive.

    Baur, Parsons and Tschabold~\cite{BPT} showed that all infinite periodic positive integral frieze patterns can be obtained in this way. Moreover, they proved the following.

  \begin{theorem}[\cite{BPT}, Theorem~4.6]
Every infinite periodic positive integral frieze pattern comes from a triangulation of an annulus or a once punctured disc.
    
    \end{theorem}

    \begin{example}
The infinite periodic frieze pattern from problem (16) of Conway and Coxeter (see Fig.~\ref{ex-sl2}) can be constructed from a once punctured disc with one marked point on the boundary (i.e., self-folded triangle).
      
      \end{example}

\section{Other results and directions}

We will conclude by giving a brief list of some  other topics arising from Conway and Coxeter's work and persistently appearing in different contexts: 

\begin{itemize}
\item {\it Difference equations, integrable systems, discrete dynamical systems}, see~\cite{MG17}, \cite{MGOST}.

\item{\it Cluster categories and categorification}, see~\cite{P23}.   

\item{\it Frieze determinants},  see~\cite{BCI}, \cite{BM1}.

\item{\it Frieze patterns and continued fractions}, see~\cite{Sm}.  
  
\end{itemize}

There are also numerous works on various modifications of frieze patterns, including: $2$-friezes, $q$-deformed friezes, symplectic friezes, non-commutative  friezes, Heronian friezes, weak friezes, additive friezes, $Y$-friezes...
See~\cite{MG2} for a detailed review, and

\href{https://www.maths.dur.ac.uk/users/anna.felikson/Projects/frieze/frieze-res.html}{https://www.maths.dur.ac.uk/users/anna.felikson/Projects/frieze/frieze-res.html}

\noindent
for more references.

\end{document}